\documentclass[table]{amsart} 

\usepackage{euscript}           
\usepackage{pstricks}
\usepackage{pst-plot}
\usepackage{float,lscape}
\input{multido.tex}
\usepackage{xcolor}

\usepackage{amsrefs,amsmath,amsthm,amssymb}            
\usepackage{graphicx,enumerate,calc,lscape}
\usepackage[matrix,arrow,curve,frame]{xy}   

\usepackage{longtable,booktabs}

\newtheorem{thm}[subsection]{Theorem}
\newtheorem{prop}[subsection]{Proposition}
\newtheorem{cor}[subsection]{Corollary}
\newtheorem{lemma}[subsection]{Lemma}

\theoremstyle{definition}  

\newtheorem{ex}[subsection]{Example}

\newtheorem{remark}[subsection]{Remark}

\newcommand{\cat}{\EuScript}    
\newcommand{\cA}{{\cat A}}      

\newcommand{\ra}{\rightarrow}                   

\newcommand{\map}{\rightarrow}

\newcommand{\field}[1]  {\mathbb #1} 
\newcommand{\A}         {\field A}
\newcommand{\F}         {\field F}
\newcommand{\R}         {\field R}

\newcommand{\Z}         {\field Z}
\newcommand{\C}         {\field C}
\newcommand{\M}         {\field M}

\DeclareMathOperator{\Ext}{Ext}

\newcommand{\sA}{\cA}

\newcommand{\dfn}{\textbf} 
\newcommand{\mdfn}[1]{\dfn{\mathversion{bold}#1}} 

\newcommand{\Smash}             {\wedge}

\numberwithin{equation}{subsection}

\newenvironment{myequation}
  {\addtocounter{subsection}{1}\begin{eqnarray}}
  {\end{eqnarray}$\!\!$}

\newcommand{\AIexc}{\textcolor{red}}
\newcommand{\same}{\cellcolor{gray!25}}

\newcommand{\hpi}{\hat{\pi}}
\newcommand{\hpimot}{\hat{\pi}^{\R}}
\newcommand{\pimot}{\pi^{\R}}
\newcommand{\hpieq}{\hat{\pi}^{\Z/2}}
\newcommand{\pieq}{\pi^{\Z/2}}
\newcommand{\Amot}{\sA_{\R}}
\newcommand{\Aeq}{\sA_{\Z/2}}
\newcommand{\Mmot}{\M^{\R}_2}
\newcommand{\Meq}{\M^{\Z/2}_2}
\newcommand{\Extmot}{\Ext_{\R}}
\newcommand{\Exteq}{\Ext_{\Z/2}}
\newcommand{\augmot}{\overline{\sA}_{\R}}
\newcommand{\augeq}{\overline{\sA}_{\Z/2}}

\newcommand{\cirradnew}{0.1}

\newrgbcolor{grey}{0.5 0.5 0.5}

\begin{document}

\title{$\Z/2$-equivariant and $\R$-motivic stable stems}

\author{Daniel Dugger}
\address{Department of Mathematics\\ University of Oregon\\ Eugene, OR
97403} 
\email{ddugger@math.uoregon.edu}

\author{Daniel C.\ Isaksen}
\address{Department of Mathematics \\ Wayne State University\\
Detroit, MI 48202}
\email{isaksen@wayne.edu}


\begin{abstract}
We establish an isomorphism between the 
stable homotopy groups $\hpimot_{s,w}$ 
of the 2-completed $\R$-motivic sphere spectrum
and the stable homotopy groups
$\hpieq_{s,w}$ of the 2-completed $\Z/2$-equivariant sphere spectrum
when $s \geq 3 w - 5$ or $s \leq -1$.
\end{abstract}

\maketitle


\section{Introduction}

This paper is a sequel to \cite{DI15}, where we 
computed some of the stable homotopy groups
of the 2-completed motivic sphere spectrum over the ground field $\R$. 
Here we explain that in a
certain range these groups
agree with the analogous $\Z/2$-equivariant (but non-motivic) stable homotopy 
groups.

There is an equivariant realization functor from $\R$-motivic stable
homotopy theory to $\Z/2$-equivariant homotopy theory,
induced by assigning to every scheme $X$ over $\R$ the associated
analytic space $X(\C)$ with complex conjugation \cite{MV99}*{Section 3.3}.  
This induces a map
\begin{myequation}
\label{eq:compare}
 \hpimot_{*,*} \ra \hpieq_{*,*}
\end{myequation}
of bigraded rings,
where the domain is the stable homotopy ring of the 2-completed
$\R$-motivic sphere spectrum, and the
target is the stable homotopy ring of the 2-completed
$\Z/2$-equivariant sphere spectrum.
Each group $\hpieq_{s,w}$ is finitely-generated,
so the $2$-completion on the right is very mild.
The reader should beware that the stable homotopy groups of the
$2$-completed motivic sphere are not necessarily the same as the
algebraic $2$-completions of the stable homotopy groups of the
uncompleted motivic sphere.  One must account for $\eta$-completion
as well \cite{HKO11b}*{Theorem 1}.
For the purposes of this paper, 
$\hpimot_{*,*}$ and $\hpieq_{*,*}$ can be defined as the objects 
to which the $\R$-motivic and $\Z/2$-equivariant Adams spectral sequences converge,
respectively.

The $\Z/2$-equivariant stable homotopy groups
were computed in a range by Araki and Iriye \cites{AI,I}, although the
method of computation and statements of results are difficult
to navigate.
A goal of the
work begun in \cite{DI15}
is to better understand the
Araki-Iriye results by lifting as much as possible 
back to $\R$-motivic homotopy theory via the map
(\ref{eq:compare}).  The present paper demonstrates that this is possible
in a range.

\subsection{Equivariant homotopy groups}

Recall that $\R^{1,1}$ denotes the real line with the sign
representation of $\Z/2$, whereas $\R^{1,0}$ denotes the real line
with the trivial representation.  For $p\geq q$ one sets 
\[ \R^{p,q}=(\R^{1,0})^{\oplus (p-q)} \oplus (\R^{1,1})^{\oplus q},
\]
and $S^{p,q}$ is the one-point compactification of $\R^{p,q}$.  These
are the bigraded $\Z/2$-equivariant spheres, and we write
$\pieq_{p,q}$ for the $\Z/2$-equivariant stable homotopy group
$[S^{p,q},S^{0,0}]$.  These groups were computed by Araki-Iriye in the
range $p\leq 13$, although the calculations for $p = 12$ and $p = 13$ 
were announced without proof \cite{AI} \cite{I}.

One way to understand the global structure
of $\pieq_{*,*}$ is to break the calculation into pieces as follows.
The \mdfn{$n$th $\Z/2$-equivariant Milnor-Witt stem} is 
the collection of groups
\[ \oplus_{p} \pieq_{p+n,p}. \]
The $0$th
Milnor-Witt stem is a subring of $\pieq_{*,*}$, and the $n$th
Milnor-Witt stem is a module over this subring.  

Table \ref{tab:pieq} at the end of the article gives some partial
results about $\Z/2$-equivariant stable homotopy groups, arranged so that
the groups in each row belong to a common Milnor-Witt stem.

We will give a global picture of the current knowledge of
$\Z/2$-equivariant stable homotopy groups.
One piece of the global structure relates to the fixed-point map
\[ \phi\colon \pieq_{s,w} \ra \pi_{s-w} \]
from the equivariant to the non-equivariant groups.
This map is known
to be split for $s \geq 2w$ \cite{Bredon67}*{p.\ 284}, 
and is an isomorphism for $s<0$ 
\cite{AI}*{Proposition 7.0}.  
These splittings are represented by copies of
$\pi_{s-w}$ in Table \ref{tab:pieq}.

The second piece of global structure consists of periodicity,
for each fixed $s$, of the kernel of  $\phi\colon \pi_{s,*}\ra
\pi_{s-*}$.  Note that when $*>s$ this is $\pi_{s,*}$ itself, whereas
when $s\geq 2w$ it is a summand (by the preceding
paragraph).  There are two difficulties with the periodicity phenomenon.
First, the orders of the periodicities and the values of the 
periodic groups
are rather complicated.  See Table 2 of \cite{I} for
a complete description in the range $s \leq 13$, and
beware that the indexing in that table differs from ours:
the correspondence is given by the equations $s = p + q$ and $w = p$.
Second, there are exceptions to the periodicity in the 
range $2w \geq s \geq w-1$
\cite{AI}*{Proposition 4.8}.  These exceptions are shown in red
in Table \ref{tab:pieq}.  Note, however, that {\it some\/} of the groups in the 
range $2w \geq s \geq w-1$ actually do assume the periodic values.

The groups $\pieq_{p,0}$ and $\pieq_{p,1}$ are also computed in
\cite{Szymik} for $p\leq 13$ using the equivariant Adams spectral
sequence based on Borel cohomology.

\subsection{Motivic homotopy groups}
The motivic setup \cite{MV99} is similar to the equivariant setup.
Now $S^{1,0}$ is the simplicial circle,
$S^{1,1}$ is the scheme $\A^1 - 0$, and
$S^{p,q}$ is the appropriate smash product of copies of
$S^{1,0}$ and $S^{1,1}$.
We use the same notation $S^{p,q}$ for motivic spheres and equivariant
spheres.  
Equivariant realization sends one to the other,
so this abuse of notation generally does not lead to confusion. 

We write $\pimot_{p,q}$ for the $\R$-motivic stable homotopy group
$[S^{p,q}, S^{0,0}]$.
The \mdfn{$n$th $\R$-motivic Milnor-Witt stem} is the collection of 
groups
\[ \oplus_{p} \pimot_{p+n,p}. \]
As in the equivariant case,
the $0$th Milnor-Witt stem is a subring, and 
the $n$th Milnor-Witt stem is a module over the $0$th Milnor-Witt stem.
Morel's connectivity theorem \cite{Morel05}
shows that the negative Milnor-Witt stems are zero.
Moreover, Morel has calculated the $0$th Milnor-Witt stem
in terms of Milnor-Witt $K$-theory \cite{Morel04}*{Section 6}.

Morel's calculation gives an explicit description
of $\pimot_{-1,-1}$, 
but it turns out to be a complicated
uncountable group.  In order to carry out further calculations,
we find it convenient to work with the stable homotopy groups
of the 2-completed $\R$-motivic sphere.  One could complete at odd primes as 
well, but we do not address that here.

We will now set aside the $\R$-motivic stable homotopy ring $\pimot_{*,*}$,
and instead work with the stable homotopy ring $\hpimot_{*,*}$
of the 2-completed $\R$-motivic sphere.  This ring splits into
Milnor-Witt stems as before.
The 2-complete negative Milnor-Witt stems are still zero,
and the 2-complete $0$th Milnor-Witt stem can be easily
described with generators and relations.
Moreover, the first, second, and third Milnor-Witt stems have been
completely described \cite{DI15}.  The authors have preliminary
data on the $n$th Milnor-Witt stems for $n \leq 15$;
these results will appear in a future article.

\subsection{The comparison}

The map (\ref{eq:compare}) is not an
isomorphism in general.
We know that the negative $\R$-motivic Milnor-Witt stems vanish,
whereas Table \ref{tab:pieq} shows that in the $\Z/2$-equivariant context the
negative
Milnor-Witt stems are non-trivial.
In the $0$th Milnor-Witt stems,
the map (\ref{eq:compare}) is an isomorphism when
$p\leq 4$ but not in general \cite{AI}*{Theorem 12.4(iii)}.  
Likewise, the computations of \cite{DI15} show that
$\hpimot_{*,*}$ vanishes in the first Milnor-Witt stem for weights
larger than $2$, whereas the $\Z/2$-equivariant analog of this is false.

Nevertheless, we find
that the map (\ref{eq:compare}) 
is an isomorphism 
in a certain range.
The following is the main result of the paper.

\begin{thm}
\label{th:main}
The realization map $\hpimot_{s,w}\ra \hpieq_{s,w}$ is
an isomorphism in the range $s\geq 3w-5$ or $s \leq -1$.
\end{thm}

In Table \ref{tab:pieq} the range from the above theorem is shaded.
All of the groups in that region
coincide, up to 2-completion, with their 2-completed $\R$-motivic analogues.

\begin{ex}
\label{ex:sigma}
We computed in \cite{DI15} that
$\hpimot_{7,4}$ contains an element of order 32.
Theorem \ref{th:main} implies that
$\hpieq_{7,4}$ also contains an element of order 32.
This is somewhat surprising because the classical image of $J$ in
the 7-stem has order 16.
In fact, this phenomenon is already apparent in the results of 
Araki and Iriye \cite{AI}.
This observation calls strongly for a more careful study of the
motivic and equivariant images of $J$.
\end{ex}

We note two immediate consequences of Theorem~\ref{th:main}.  First,
consider the map $\hpimot_{s,w} \ra \hpi_{s-w}$ induced by taking
fixed points of equivariant realization.  Theorem \ref{th:main}
implies that this map is an isomorphism in the range $s \leq -1$ and a
split surjection for $s\geq \max\{3w-5,2w\}$, based on the analogous facts
 for $\phi\colon \pieq_{s,w} \ra
\pi_{s-w}$.  
Secondly, the known periodicity phenomena in the $\pieq_{s,*}$ groups
can now be transplanted into the $\R$-motivic context,
as in Corollary \ref{cor:period}.

\begin{cor}
\label{cor:period}
For fixed $s$ in the range $s\geq \max\{3w-5,2w\}$, 
the complementary summands of $\hpi_{s-w}$ in 
$\hpimot_{s,w}$ are periodic in $w$.
\end{cor}

We do not give the periods in Corollary \ref{cor:period},
but specific formulas for these are known from the
equivariant context.

Corollary \ref{cor:period} describes a qualitative property of
$\R$-motivic stable homotopy groups that deserves further study and 
is related to $\tau^{2^n}$-periodic families
in the $\R$-motivic Adams spectral sequence (see \cite{DI15} for an
introduction to this basic phenomenon).
We expect to return to the topic of motivic periodicity in future work.

The proof of Theorem~\ref{th:main} is straightforward.  
Equivariant realization induces a map
from the 
$\R$-motivic Adams spectral sequence to the
$\Z/2$-equivariant Adams spectral sequence.
The $\R$-motivic and $\Z/2$-equivariant Steenrod algebras agree in a range
of dimensions.  This gives an isomorphism on cobar complexes in a range,
which shows that $\R$-motivic and $\Z/2$-equivariant $\Ext$ groups agree
in a range.  In other words, the $\Z/2$-equivariant and $\R$-motivic Adams
$E_2$-pages agree in a range.
Finally, this induces an isomorphism in homotopy groups in a range.
The only complications arise as matters of bookkeeping.

\subsection{Notation}

For the reader's convenience, we record here notation used
in the article.
\begin{itemize}
\item
$\Mmot$ is the $\R$-motivic homology of a point with $\F_2$ coefficients.
\item
$\Meq$ is the $\Z/2$-equivariant homology of a point with $\F_2$ coefficients.
\item
$\Amot$ is the dual $\R$-motivic Steenrod algebra.  We grade elements in the form
$(t,w)$, where $t$ is the internal Steenrod degree and $w$ is the
motivic weight.
\item
$\Aeq$ is the dual $\Z/2$-equivariant Steenrod algebra.  We grade elements in the form
$(t,w)$, where $t$ is the internal Steenrod degree and $w$ is the
equivariant weight.
\item
$\augmot$ is the augmentation ideal of $\Amot$.
\item
$\augeq$ is the augmentation ideal of $\Aeq$.
\item
$C_\R^*$ is the $\R$-motivic cobar complex.
\item
$C_{\Z/2}^*$ is the $\Z/2$-equivariant cobar complex.
\item
$\Extmot = \Ext_{\Amot} ( \Mmot, \Mmot)$ is the cohomology of the $\R$-motivic 
Steenrod algebra.
We grade elements in the form
$(s,f,w)$, where $s = t - f$ is the stem, $f$ is the Adams filtration,
and $w$ is the motivic weight.
\item
$\Exteq = \Ext_{\Aeq} ( \Meq, \Meq)$ is the cohomology of the $\Z/2$-equivariant 
Steenrod algebra.
We grade elements in the form
$(s,f,w)$, where $s = t - f$ is the stem, $f$ is the Adams filtration,
and $w$ is the equivariant weight.
\item
$\hpimot_{*,*}$ is the stable homotopy ring of the 2-completed $\R$-motivic
sphere.  We grade elements in the form $(s,w)$, where $s$ is the stem
and $w$ is the motivic weight.
\item
$\hpieq_{*,*}$ is the stable homotopy ring of the 2-completed $\Z/2$-equivariant
sphere.  We grade elements in the form $(s,w)$, where $s$ is the stem
and $w$ is the equivariant weight.
\end{itemize}

For sake
of tradition, we refer to $\Amot$ and $\Aeq$ as Steenrod algebras.
More precisely, $(\Mmot, \Amot)$ and $(\Meq, \Aeq)$ are Hopf algebroids,
not Hopf algebras, because $\Mmot$ is a non-trivial $\Amot$-module,
and $\Meq$ is a non-trivial $\Aeq$-module.  

\section{The motivic and equivariant Steenrod algebras}

Let $H^{\R}$ denote the $\R$-motivic Eilenberg-MacLane spectrum representing
motivic cohomology with $\F_2$ coefficients, and let
$\Mmot =\pi_{*,*}(H^{\R})$
be the homology of a point.
Recall that
$\Mmot$ equals $\F_2[\tau,\rho]$ where $\tau$ has homological degree $(0,-1)$
and $\rho$ has homological degree $(-1,-1)$ \cite{Voevodsky03b}.

Let $\Amot = \pimot_{*,*}(H^{\R} \Smash H^{\R})$ be the dual $\R$-motivic
Steenrod algebra.
Recall that $\Amot$ is equal to 
\[ 
\M_2[\tau_0,\tau_1,\ldots,\xi_0,\xi_1,\ldots]/(\xi_0=1,\tau_k^2=\tau
\xi_{k+1}+\rho\tau_{k+1}+\rho\tau_0\xi_{k+1}),
\]
where $\xi_i$ has bidegree $(2(2^i-1),2^i-1)$ and $\tau_i$ has
bidegree $(2^{i+1}-1,2^i-1)$ \cite{Voevodsky10}.
For a summary of the complete Hopf algebroid structure,
see \cite{DI15}*{Section 2}.
Observe that $\Amot$ is free as a
left $\Mmot$-module, with basis given by monomials
$\tau_{0}^{\epsilon_0}\tau_1^{\epsilon_1}\ldots
\tau_r^{\epsilon_r}\xi_1^{n_1}\ldots \xi_s^{n_s}$ where 
$0 \leq \epsilon_i \leq 1$ and $n_i \geq 0$.
We abbreviate such a monomial as $\tau^\epsilon \xi^n$.

When we build the cobar complex, we will use the augmentation
ideal $\augmot$ of $\Amot$, i.e., the kernel of the augmentation map
$\Amot \map \Mmot$.
Observe that $\augmot$ is also free as a 
left $\Mmot$-module, with 
the same basis as for $\Amot$ except that the monomial $1$
is excluded.

Similarly, let $H^{\Z/2}$ denote the $\Z/2$-equivariant Eilenberg-MacLane spectrum
corresponding to the constant Mackey functor with value $\F_2$.  Write
$\Meq=\pi_{*,*}(H^{\Z/2})$ for the coefficient ring 
and $\Aeq=\pi_{*,*}(H^{\Z/2} \Smash H^{\Z/2})$ for the $\Z/2$-equivariant
dual Steenrod algebra. 

We will now recall an explicit description of $\Meq$ 
\cite{HK01}*{Proposition 6.2}.
It contains $\Mmot$ as a subring, but also contains a ``dual copy'' in
opposing dimensions. 
Figure \ref{fig:Meq} gives a complete description of
$\Meq$.
Every dot
denotes a copy of $\F_2$, vertical lines represent multiplication by
$\tau$, and diagonal lines represent multiplication by $\rho$.

\begin{figure}[h]
\label{fig:Meq}
\begin{pspicture}(-5,-2.25)(5,3.5)
\psset{unit=0.5}
\psline[linecolor=grey]{->}(-5,0)(5,0)
\psline[linecolor=grey]{->}(0,-4)(0,7)

\multido{\i=0+-1}{5}
{\pscircle*(0,\i){\cirradnew}}

\multido{\i=-1+-1}{4}
{\pscircle*(-1,\i){\cirradnew}}

\multido{\i=-2+-1}{3}
{\pscircle*(-2,\i){\cirradnew}}

\multido{\i=-3+-1}{2}
{\pscircle*(-3,\i){\cirradnew}}

\pscircle*(-4,-4){\cirradnew}

\psline{->}(0,0)(0,-4.5)
\psline{->}(-1,-1)(-1,-4.5)
\psline{->}(-2,-2)(-2,-4.5)
\psline{->}(-3,-3)(-3,-4.5)
\psline{->}(-4,-4)(-4,-4.5)
\psline{->}(0,0)(-4.5,-4.5)
\psline{->}(0,-1)(-3.5,-4.5)
\psline{->}(0,-2)(-2.5,-4.5)
\psline{->}(0,-3)(-1.5,-4.5)
\psline{->}(0,-4)(-0.5,-4.5)
\multido{\i=2+1}{5}
{\pscircle*(0,\i){\cirradnew}}
\multido{\i=3+1}{4}
{\pscircle*(1,\i){\cirradnew}}
\multido{\i=4+1}{3}
{\pscircle*(2,\i){\cirradnew}}
\multido{\i=5+1}{2}
{\pscircle*(3,\i){\cirradnew}}
\pscircle*(4,6){\cirradnew}

\psline{->}(0,2)(0,6.5)
\psline{->}(1,3)(1,6.5)
\psline{->}(2,4)(2,6.5)
\psline{->}(3,5)(3,6.5)
\psline{->}(4,6)(4,6.5)
\psline{->}(0,2)(4.5,6.5)
\psline{->}(0,3)(3.5,6.5)
\psline{->}(0,4)(2.5,6.5)
\psline{->}(0,5)(1.5,6.5)
\scriptsize
\put(0.2,0.2){$1$}
\put(0.2,-1){$\tau$}
\put(-1.6,-1){$\rho$}
\put(-0.5,1.8){$\theta$}
\put(4.5,0.2){$t$}
\put(-0.6,6.7){$w$}
\end{pspicture}
\caption{The equivariant coefficient ring $\Meq$ (homological grading)}
\end{figure}
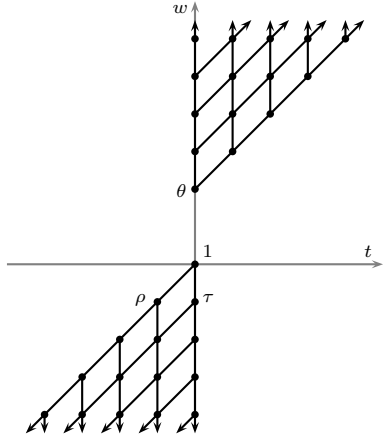

In words, $\Meq$ in bidegree $(t,w)$
consists of a copy of $\F_2$ when:
\begin{enumerate}
\item
$t \geq 0$ and $w \geq t + 2$, or
\item
$t \leq 0$ and $w \leq t$.
\end{enumerate}
The element in bidegree $(0,2)$ is called $\theta$, 
and the other elements
in the ``dual copy'' are typically named
$\frac{\theta}{\tau^k\rho^l}$ for $k \geq 0$ and $l\geq 0$.  
This naming convention
respects the product structure, although one must remember that
neither $\tau$ nor $\rho$ is actually invertible.  
Any two elements of the form
$\frac{\theta}{\tau^k\rho^l}$ multiply to zero.
These details about the product structure will not be needed in our
analysis.  

The dual $\Z/2$-equivariant 
Steenrod algebra $\Aeq$ has the same description as the
$\R$-motivic Steenrod algebra, but with $\Mmot$ replaced with $\Meq$
\cite{HK01}*{Theorem 6.41}.
In particular, note that $\Aeq$ is free as a left $\Meq$-module on
the same basis.
More explicitly,
$\Aeq$ is equal to $\Meq \otimes_{\Mmot} \Amot$.
The augmentation
ideal $\augeq$ of $\Aeq$
is also free as a 
left $\Meq$-module, with 
the same basis as for $\Aeq$ except that the monomial $1$
is excluded, so $\augeq$ is equal to
$\Meq \otimes_{\Mmot} \augmot$.

Equivariant realization from $\R$-motivic homotopy
theory to $\Z/2$-equivariant homotopy theory 
yields a map $(\Amot,\Mmot) \ra (\Aeq,\Meq)$ of Hopf algebroids.
This map is just the evident inclusion of $\Mmot$
into $\Meq$, and of $\Amot$ into $\Aeq$.  

\begin{lemma}
\label{lem:monomial-degree}
Let $\tau^\epsilon \xi^n$ have bidegree $(t,w)$.
Then $t \leq 3w + 1$.
\end{lemma}

\begin{proof}
The bidegree of each $\xi_i$ satisfies the inequality $t \leq 3w$.
Similarly, if $i \geq 1$, then the bidegree of $\tau_i$ 
also satisfies $t \leq 3w$.
Therefore the bidegree of $\tau^\epsilon \xi^n$
satisfies the inequality $t \leq 3w$ if $\epsilon_0 = 0$.

On the other hand, if $\epsilon_0 = 1$, then write
$\tau^\epsilon \xi^n$ as $\tau_0 \tau^{\epsilon'} \xi^n$,
where $\epsilon'_0 = 0$.
The bidegree of $\tau^{\epsilon'} \xi^n$ satisfies the inequality
$t \leq 3w$, so the bidegree of
$\tau_0 \tau^{\epsilon'} \xi^n$ satisfies $t \leq 3w + 1$.
\end{proof}

\begin{remark}
In fact, one can make a much stronger statement about the bidegrees of
the elements $\tau^\epsilon \xi^n$.  In general,
the bidegree of such an element satisfies the inequality
$t \geq 2^{c+1} - c - 2$, where $c = t - 2w$ is the ``Chow degree''.
However, this stronger inequality does not end up yielding a stronger result
about stable homotopy groups.  Likewise, the result in the following
lemma is non-optimal---but the slope of $3$ is chosen precisely because it
interacts well with the bound from the previous lemma.  
\end{remark}

\begin{lemma}
\label{lem:theta-degree}
Let $(t,w)$ be the bidegree of the element $\frac{\theta}{\rho^k
  \tau^l}$ 
in $\Meq$.  Then one has $t \leq 3w - 6$.
\end{lemma}

\begin{proof}
In Figure \ref{fig:Meq},
the elements of the form $\frac{\theta}{\rho^k \tau^l}$ 
all lie on or above the line $t = 3w - 6$.
\end{proof}

\section{Cobar complexes and $\Ext$ groups}

Next we proceed to the cobar complexes of $\Amot$ and $\Aeq$, respectively.
These cobar complexes are differential graded algebras whose homologies
give the $\R$-motivic and $\Z/2$-equivariant $\Ext$ groups.
We will obtain an isomorphism of $\Ext$ groups in a range by establishing
an isomorphism of cobar complexes in a range.

Let $C_\R^*$ and $C_{\Z/2}^*$ be the $\R$-motivic and $\Z/2$-equivariant
cobar complexes.  By definition, $C_\R^f$ is equal to
\[
\augmot \otimes_{\Mmot} \augmot \otimes_{\Mmot} \cdots \otimes_{\Mmot}
\augmot,
\]
where there are $f$ factors in the tensor product.  Similarly,
$C_{\Z/2}^f$ is equal to
\[
\augeq \otimes_{\Meq} \augeq \otimes_{\Meq} \cdots \otimes_{\Meq}
\augeq.
\]

\begin{lemma}
\label{lem:cobar}
The $\Z/2$-equivariant cobar complex $C_{\Z/2}^*$ is isomorphic to
$\Meq \otimes_{\Mmot} C_{\R}^*$.
\end{lemma}

\begin{proof}
Use that $\augeq$ is equal to $\Meq \otimes_{\Mmot} \augmot$.  Then
$C_{\Z/2}^f$ can be rewritten as $\Meq \otimes_{\Mmot} C_{\R}^f$.
\end{proof}

\begin{lemma}
\label{lem:cobar-iso}
The map $C_\R^f \map C_{\Z/2}^f$
is:
\begin{itemize}
\item
an injection in all degrees.
\item
an isomorphism in degrees satisfying $t - f \geq 3w - 5$.
\item
an isomorphism in degrees satisfying $t \leq f - 1$.
\end{itemize}
\end{lemma}

\begin{proof}
By Lemma \ref{lem:cobar}, 
we are considering the obvious map $C_\R^f \map \Meq \otimes_{\Mmot} C_\R^f$.
Since the map $\Mmot \map \Meq$ is injective and $C_\R^f$ is free over
$\M_2^\R$, it follows that 
the map 
$C_\R^f \map \Meq \otimes_{\Mmot} C_\R^f$
is injective for every $f \geq 0$.

Consider a typical element
$\frac{\theta}{\rho^k \tau^l} [ \zeta_1 | \cdots | \zeta_f ]$
of the cokernel of the map.
By Lemma \ref{lem:monomial-degree},
each $\zeta_i$ has bidegree $(t_i, w_i)$
satisfying $t_i \leq 3 w_i + 1$.
Summing over $i$, we obtain that
$[ \zeta_1 | \cdots | \zeta_f ]$ has bidegree satisfying
$t \leq 3 w + f$.
Finally, 
Lemma \ref{lem:theta-degree}
implies that the bidegree of
$\frac{\theta}{\rho^k \tau^l} [ \zeta_1 | \cdots | \zeta_f ]$
satisfies $t \leq 3w + f - 6$.
Therefore, the cokernel vanishes in bidegrees satisfying
$t - f \geq 3w - 5$.

Similarly,
each $\zeta_i$ has bidegree $(t_i, w_i)$
satisfying $t_i \geq 1$, so
$[ \zeta_1 | \cdots | \zeta_f ]$ has bidegree satisfying
$t \geq f$.
Then the bidegree of
$\frac{\theta}{\rho^k \tau^l} [ \zeta_1 | \cdots | \zeta_f ]$
also satisfies $t \geq f$.
Therefore, the cokernel vanishes in bidegrees satisfying
$t \leq f -1$.
\end{proof}

\begin{remark}
\label{rem:cobar-sharp}
The inequalities in Lemma \ref{lem:cobar-iso} are sharp in the following 
sense.
The element
$\theta [ \tau_0 \tau_1 | \tau_0 \tau_1 | \cdots | \tau_0 \tau_1 ]$
of $C_{\Z/2}^f$ lies on the line $t - f = 3w - 6$, and it 
does not belong to the image of $C_\R^f$.
Also, the element
$\theta [ \tau_0 | \tau_0 | \cdots | \tau_0 ]$ of 
$C_{\Z/2}^f$ lies on the line $t = f$, and it
does not belong to the image of $C_\R^f$.
\end{remark}

The following lemma from homological algebra
will let us deduce an $\Ext$ isomorphism
from the cobar isomorphism of Lemma \ref{lem:cobar-iso}.
The two parts are dual, and the proofs are simple diagram chases.

\begin{lemma}
\label{le:chain}
Let $C_*\ra D_*$ be a map of homologically graded chain complexes
(so the differentials decrease degree).
\begin{enumerate}[(a)]
\item
Suppose that $C_i\ra D_i$ is an
isomorphism for all $i\geq n+1$, and an injection for $i=n$.  Then
the map $H_i(C) \ra H_i(D)$ of homology groups is:
\begin{itemize}
\item
an injection for $i=n$,
\item
an isomorphism for $i\geq n+1$.
\end{itemize}
\item Dually,
suppose that $C_i\ra D_i$ is an
isomorphism for all $i\leq n-1$, and a surjection for $i=n$.  Then
the map $H_i(C) \ra H_i(D)$ of homology groups 
\begin{itemize}
\item
is an isomorphism for $i\leq n-1$,
\item
a surjection for $i=n$.
\end{itemize}
\end{enumerate}
\end{lemma}

We will use the grading $(s,f,w)$ for $\Ext$ groups, where
$s$ is the stem, $f$ is the Adams filtration, and $w$ is the weight.
An element of degree $(s,f,w)$ occurs at Cartesian coordinates $(s,f)$
in a standard Adams chart.  Recall that $s = t - f$, where $t$
is the internal Steenrod degree.

\begin{prop}
\label{prop:Ext-iso}
In degree $(s,f,w)$, the map $\Extmot \ra \Exteq$ is:
\begin{itemize}
\item
an injection if $s=3w-6$.  
\item
an isomorphism if $s \geq 3w-5$.
\end{itemize}
\end{prop}

\begin{proof}
The claims follow immediately from Lemmas \ref{lem:cobar-iso}
and \ref{le:chain} because $\Ext$ can be computed as 
the homology of the cobar construction.
\end{proof}

\begin{prop}
\label{prop:Ext-iso2}
In degree $(s,f,w)$, the map $\Extmot \ra \Exteq$ is
an isomorphism if $s \leq -1$.
\end{prop}

\begin{proof}
Lemmas \ref{lem:cobar-iso} and \ref{le:chain}
imply that the map is an isomorphism if $s \leq -2$
and is a surjection if $s= -1$.
In order to obtain the isomorphism for $s = -1$, we need to 
investigate the cobar complex a little further.

In degrees satisfying $s = 0$, i.e., $t = f$, 
the cokernel of the map
$C_\R^* \map C_{\Z/2}^*$ of cobar complexes
consists
elements of the form
$\frac{\theta}{\tau^a} [ \tau_0 | \tau_0 | \cdots | \tau_0 ]$.
All of these elements are cycles in the $\Z/2$-equivariant cobar complex.
A diagram chase now shows that
$\Extmot \ra \Exteq$ is an isomorphism if $s = -1$.
\end{proof}

The following finiteness condition for $\Extmot$ implies that
there are only finitely many Adams differentials in any given
degree.  We will need this fact in Section \ref{sctn:homotopy}
when we analyze the Adams spectral sequence.

\begin{lemma}
\label{lem:Ext-finite}
In each degree $(s,f,w)$, 
the group
$\Extmot^{(s,f,w)}$ is a finite-dimensional $\F_2$-vector space.
\end{lemma}

\begin{proof}
As described in \cite{DI15}*{Section 3}, there is a $\rho$-Bockstein
spectral sequence converging to $\Extmot$.  
It suffices to show that the $E_1$-page of this spectral sequence is
finite-dimensional over $\F_2$ in each tridegree.
In degree $(s,f,w)$,
this $E_1$-page consists of elements of the form $\rho^k x$,
where $k\geq 0$ and $x$ belongs to the $\C$-motivic $\Ext$ group in degree
$(s+k, f, w+k)$.

The $\C$-motivic $\Ext$ groups have a vanishing plane,
as described in \cite{DI15}*{Lemma 2.2}.
In this case, the vanishing plane implies that 
$k \leq s+f-2w$ if $x$ is non-zero.  Since $k$ is non-negative this means there
are only finitely-many values of $k$ that contribute to the $E_1$-page
of our spectral sequence in degree $(s,f,w)$.  

Finally, the $\C$-motivic $\Ext$ groups are degreewise finite-dimensional.
This follows from the fact that the $E_1$-page of the motivic May spectral
sequence is degreewise finite-dimensional.
\end{proof}

\section{Homotopy groups}
\label{sctn:homotopy}

We now come to our main results comparing $\R$-motivic and 
$\Z/2$-equivariant homotopy groups.

\begin{thm}
\label{thm:pi}
The map $\hpimot_{s,w} \ra \hpieq_{s,w}$ is:
\begin{itemize}
\item
an injection if $s=3w-6$.
\item
an isomorphism if $s\geq 3w-5$.
\end{itemize}
\end{thm}

\begin{proof}
Proposition \ref{prop:Ext-iso}
gives an isomorphism (in a range) between the $E_2$-pages of the
$\R$-motivic and $\Z/2$-equivariant Adams spectral sequences.
Inductively, 
Lemma \ref{le:chain} gives isomorphisms (in a range)
between the $E_r$-pages of the spectral sequences for all $r$.
The finiteness condition of Lemma \ref{lem:Ext-finite}
guarantees that for each degree $(s,f,w)$,
there exists an $r$ such that 
the $E_\infty$-page is isomorphic to the $E_r$-page.
Therefore, we obtain an isomorphism of $E_\infty$-pages
in a range.

The $E_\infty$-pages are associated graded objects of 
the stable homotopy groups.
This implies
that the stable homotopy groups are isomorphic as well.

The same style of argument applies to the claim about injections.
\end{proof}

\begin{thm}
\label{thm:pi2}
The map $\hpimot_{s,w} \ra \hpieq_{s,w}$ is
an isomorphism if $s \leq -1$.
\end{thm}

\begin{proof}
The argument from Theorem \ref{thm:pi}
implies that the map is an isomorphism for $s \leq -2$
and a surjection for $s = -1$.  In order to obtain the isomorphism
for $s = -1$, we have to investigate the Adams $E_2$-pages slightly further.

Recall from the proof of Proposition \ref{prop:Ext-iso2}
that in degrees satisfying $s = 0$, 
the cokernel of the map 
$C_\R^* \map C_{\Z/2}^*$ of cobar complexes
consists of elements of the
form $\frac{\theta}{\tau^k} [ \tau_0 | \tau_0 | \cdots | \tau_0 ]$.
Therefore, in degrees satisfying $s = 0$, the cokernel of the map
$\Extmot \map \Exteq$ consists of elements of the form
$\frac{\theta}{\tau^k} h_0^i$.  These elements are all
permanent cycles in the $\Z/2$-equivariant Adams spectral sequence.
In other words, there is a one-to-one correspondence between
$\R$-motivic and $\Z/2$-equivariant Adams differentials from the 0-stem to the
$(-1)$-stem.

A diagram chase now establishes that the
$\R$-motivic and $\Z/2$-equivariant $E_\infty$-pages are isomorphic for $s = -1$.
This passes to an isomorphism of stable homotopy groups.
\end{proof}

We restate Theorem \ref{thm:pi} 
in an equivalent form that is useful
from the Milnor-Witt degree perspective.

\begin{cor}
\label{cor:MW-iso}
On the $n$th Milnor-Witt stems, the map
$\hpimot_{*,*} \map \hpieq_{*,*}$ is:
\begin{itemize}
\item
an isomorphism in stem $s$ if $2 s \leq 3 n + 5$.
\item
an injection in stem $s$ if $2 s = 3 n + 6$.
\end{itemize}
\end{cor}

\begin{proof}
This is a straightforward algebraic rearrangement of Theorem \ref{thm:pi},
using that $n = s-w$.
\end{proof}


%

\section{Equivariant stable homotopy groups}

Table \ref{tab:pieq} summarizes some of the calculations
of Araki and Iriye \cites{AI,I}.  
The indices across the top indicate the stem $s$, 
while the indices at the left indicate the
Milnor-Witt degree $s-w$.
The $\R$-motivic and $\Z/2$-equivariant
stable homotopy groups are isomorphic in the shaded region, as described in
Theorem \ref{th:main}.

For compactness, we use the following notation to indicate abelian 
groups:
\begin{enumerate}
\item
$\infty = \Z$.
\item
$n=\Z/n$.
\item
$n \cdot m = \Z/n \oplus \Z/m$.
\item
$n^k = (\Z/n)^k$
\end{enumerate}

The symbols $\pi_k$ indicate that the classical stable homotopy
group $\pi_k$ splits via the fixed point map.

Table \ref{tab:pieq} is a companion to \cite{I}*{Table 2}, which 
gives the values of the periodic summands.
The red symbols in Table \ref{tab:pieq} are exceptions to the
periodicity.

\begin{table}[h]
\caption{Some values of $\pieq_{s,w}$}
\label{tab:pieq}

\begin{picture}(100,1)
\put(-110,-14){$\scriptstyle{s}$}
\put(-127,-20){$\scriptstyle{s-w}$}
\put(-123,-10){\line(2,-1){18}}
\end{picture}

\begin{longtable}{r|lllllllll}
& $-2$ & $-1$ & $0$ & $1$ & $2$ & $3$ & $4$ & $5$ & $6$  \\
 \midrule
 $7$ & \same $\pi_7$ & \same $\pi_7$ &  \same $2 \cdot \pi_7$  &  \same $4 \cdot \pi_7$  &  \same $8 \cdot \pi_7$  &  \same $48 \cdot 4 \cdot \pi_7$  &  \same $16 \cdot \pi_7$  &  \same $16 \cdot \pi_7$  & \same  $16 \cdot 2 \cdot \pi_7$ \\
$6$ & \same $\pi_6$ & \same $\pi_6$ & \same  $\infty \cdot \pi_6$  & \same $2 \cdot \pi_6$ & \same  $2^2 \cdot \pi_6$  & \same  $2^2 \cdot \pi_6$  & \same  $2 \cdot \pi_6$  & \same  $2 \cdot \pi_6$  & \same  $2^2 \cdot \pi_6$  \\
$5$ & \same $\pi_5$ & \same $\pi_5$ & \same  $2 \cdot \pi_5$  & \same  $2^2 \cdot \pi_5$  & \same  $2 \cdot \pi_5$  & \same  $12 \cdot \pi_5$  & \same  $\pi_5$  & \same  $\pi_5$  & \same  $\pi_5$  \\
$4$  & \same $\pi_4$ & \same $\pi_4$ & \same  $\infty \cdot \pi_4$  & \same  $\pi_4$  & \same  $\pi_4$  & \same  $2 \cdot \pi_4$  & \same  $2 \cdot \pi_4$  & \same  $2 \cdot \pi_4$  & \same  $4 \cdot \pi_4$ \\
$3$  & \same $\pi_3$ & \same  $\pi_3$  & \same  $2 \cdot \pi_3$  & \same  $4 \cdot \pi_3$  & \same  $8 \cdot \pi_3$  & \same  $24 \cdot 8 \cdot \pi_3$  & \same  $8 \cdot \pi_3$  & \same  $8 \cdot \pi_3$  & \same  $8 \cdot \pi_3$ \\
$2$ &  \same{$\pi_2$}  &  \same{$\pi_2$}  & \same $\infty \cdot \pi_2$  & \same $2 \cdot \pi_2$  & \same $2^2 \cdot \pi_2$  & \same $2 \cdot \pi_2$  &  \same $\pi_2$  &  \same \AIexc{$2$}  &  \AIexc{$2$} \\
$1$ &  \same{$\pi_1$}  &  \same{$\pi_1$}  & \same  $2 \cdot \pi_1$  & \same  $2^2 \cdot \pi_1$  & \same  $2 \cdot \pi_1$  & \same  \AIexc{$24$}  & \same  $0$  &  $0$  &  \AIexc{$0$} \\
$0$  &  \same{$\pi_0$}  &  \same{$\pi_0$}  &  \same $\infty \cdot \pi_0$  &  \same \AIexc{$\infty$}  &  \same \AIexc{$\infty$}  &  \AIexc{$\infty$}  &  \AIexc{$\infty$}  &  \AIexc{$\infty$}  &  \AIexc{$\infty \cdot 2$}  \\
$-1$ & \same{$0$}  & \same{$0$}  &  \same \AIexc{$0$}  &  \same \AIexc{$0$}  &  \AIexc{$0$}  &  \AIexc{$12$}  &  \AIexc{$0$}  &  \AIexc{$0$}  &  \AIexc{$2$}  \\
$-2$ & \same{$0$}  & \same{$0$}  &  $\infty$  &  $2$  &  $2^2$  &  $2^2$  &  $2$  &  $2$  &  $2^2$ \\
$-3$ & \same{$0$}  & \same{$0$} &   $2$  &  $2^2$  &  $2$  &  $12$  &  $0$  &  $0$  &  $0$ \\
\end{longtable}

\begin{longtable}{r|lllllll}
& $7$ & $8$ & $9$ & $10$ & $11$ & $12$ & $13$ \\
\midrule
$7$ & \same  $240 \! \cdot \! 16 \! \cdot \! 2 \! \cdot \! \pi_7$  & \same  $16 \cdot 2 \cdot \pi_7$  & \same  $16 \cdot 2 \cdot \pi_7$  & \same  $16 \cdot \pi_7$ & \same $2016 \! \cdot \! 4 \! \cdot \! \pi_7$  & \same $16 \cdot \pi_7$ & \same $48$ \\
$6$ &  \same $2 \cdot \pi_6$  & \same  $4 \cdot \pi_6$  & \same  $4 \cdot 2 \cdot \pi_6$  & \same  $6 \cdot 2^2 \cdot \pi_6$ & \same $2^2 \cdot \pi_6$  &  $\AIexc{2^2} \cdot \pi_6$  &  \AIexc{$2^2$}  \\
$5$ &  \same $240 \cdot \pi_5$  & \same  $2^3 \cdot \pi_5$  & \same  $2^5 \cdot \pi_5$  & \same  $2^2 \cdot \pi_5$ & $504$  &  $0$  &  $3$  \\
$4$ &  \same $2^2 \cdot \pi_4$  &  \same $2^4 \cdot \pi_4$  &  $2^2$  &  $3$ & $0$  &  $0$  &  $0$  \\
$3$ &  \same \AIexc{$480 \cdot 12 \cdot 4$}  &  \AIexc{$24 \cdot 4$}  &  \AIexc{$24 \cdot 2$}  &  \AIexc{$24$} & \AIexc{$504 \cdot 24$}  &  \AIexc{$24$}  &  \AIexc{$24 \cdot 3$}  \\
$2$ &  \AIexc{$0$}  &  \AIexc{$0$}  &  \AIexc{$2$}  &  \AIexc{$6 \cdot 2$} & \AIexc{$2$}  &  \AIexc{$0$}  &  \AIexc{$0$}  \\
$1$ &  \AIexc{$240$}  &  $2^3$  &  \AIexc{$2^6$}  &  \AIexc{$2^3$} & \AIexc{$504 \cdot 2$}  &  \AIexc{$2$}  &  \AIexc{$6$}  \\
$0$&  \AIexc{$\infty \cdot 2^2$}  &  \AIexc{$\infty \cdot 2^4$}  &  \AIexc{$\infty \cdot 2^2$}  &  \AIexc{$\infty \cdot 3$} & \AIexc{$\infty$}  &  \AIexc{$\infty$}  & \AIexc{$\infty$} \\
$-1$ &  \AIexc{$120 \cdot 2$}  &  \AIexc{$2$}  &  \AIexc{$2$}  &  \AIexc{$0$} & \AIexc{$252$}  & \AIexc{$0$} & \AIexc{$3$} \\
$-2$ &  $2^2$  &  $4^2$  &  $8 \cdot 4 \cdot 2$  &  $24 \cdot 2^3$ & $16 \cdot 2^2$ & $16 \cdot 2$ & $16 \cdot 2$ \\
$-3$ &  $240$  &  $2^3$  &  $2^5$  & $2^2$ & $504$ & $0$ & $3$ \\
\end{longtable}

\end{table}

\end{document}